\documentclass[12pt]{article}
\usepackage{latexsym}
\usepackage{amsfonts}
\usepackage{amsmath}
\usepackage{amsmath,amssymb,proof1,latexsym}
\usepackage[all]{xypic}

\makeindex

\newcommand{\imp}{\!\rightarrow\!}

\newcommand{\mpn}{\medskip\par\noindent}

\newcommand{\pa}{{\sf PA}}

\newcommand{\proves}{\vdash}

\newcommand{\gn}[1]{\ulcorner {#1} \urcorner }

\newcommand{\lc}[1]{#1\!\!:\!\!}

\newtheorem{Prop}{\bf Proposition}
\newenvironment{proposition}{\begin{Prop}\em }{\end{Prop}}
\newtheorem{Theor}{\bf Theorem}

\newtheorem{Lemma}{\bf Lemma}

\newtheorem{Coro}{\bf Corollary}

\newtheorem{Fact}{\bf Fact.}

\newtheorem{Remark}{\bf Remark}

\newtheorem{Claim}[enumi]{Claim}

\newtheorem{defin}{\bf Definition}

\newtheorem{exam}{\bf Example}

\newtheorem{notat}{\bf Notation.}

\newenvironment{proof}{{\bf Proof.}}{\hfill $\slot$}
\newcommand{\slot}{\hfill \mbox{$\Box$}\vspace{\parskip}\\}
\newtheorem{Comment}{\bf Comment}

\begin{document}

\title{Non-Compact Proofs}

\author{Sergei Artemov\\ \\
 {\small The Graduate Center, the City University of New York}\\
{\small  365 Fifth Avenue, New York City, NY 10016}\\
{\small {\tt sartemov@gc.cuny.edu}} }
\date{\today}
\maketitle

\begin{abstract}
Non-compact proofs are a class of reasoning that is used in mathematics but overlooked in the analysis of (un)provability of consistency. We focus on proofs of arithmetical statements (*) {\it ``for any natural number $n, F(n)$."}  A proof of (*) is called {\it compact} if all proofs of $F(n)$'s for $n=0,1,2,\ldots$ in it fit into some finitely axiomatized fragment of Peano Arithmetic \pa. An example of {\it non-compact} reasoning is given by the standard proof of 1952 Mostowski's reflexivity theorem: \pa\ proves the consistency of its finite fragments. It turns out that G\"odel's second incompleteness theorem prohibits compact proofs and does not rule out non-compact proofs of \pa-consistency formalizable in \pa. This explains how the proof of \pa-consistency in \pa\ from \cite{Art20,Art25} works. It essentially formalizes in \pa\ an appropriate explicit version of Mostowski's proof, which is not in the scope of G\"odel's theorem.
\end{abstract}

\section{Introduction}\label{intro}
We consider contentual proofs of arithmetical statements
\begin{equation}\label{examples}
 \mbox{\it ``for any natural number $n, F(n)$"} 
\end{equation}
that use only assumptions corresponding to the axioms of \pa. 
\begin{quote} {\it  Let {\em (\ref{examples})} be the claim that there are infinitely many twin primes. Imagine somebody offers a method (algorithm) $\cal M$ that, given $n$, produces an example of twin primes greater than $n$ together with a proof that ${\cal M}(n)$ works for all $n$. This would be an acceptable solution.}
\end{quote}
This solution is finite, since both the algorithm ${\cal M}$ and its verifier are finite. Assuming that it does not use principles outside the ones corresponding to \pa, it is naturally formalizable in \pa. According to the description, its step-by-step formalization is a derivation in \pa\  of 
\begin{equation}\label{form}
 \forall x \mbox{\it ``${\cal M}(x)$ is a proof of $F(x)$."} 
 \end{equation}
Although (\ref{form}) yields \pa\ proofs of $F(n)$ for any natural number $n$, it does not necessarily justify a \pa-derivation of the formal version of (\ref{examples})
\begin{equation}\label{fprime}
\forall x F(x).
\end{equation}
Indeed, a hypothetical \pa-proof of (\ref{fprime}) is a derivation in a finite fragment of \pa\ in which we can then derive all $F(n)$ for $n=0,1,2,\ldots$, whereas the proofs ${\cal M}(n)$ might require access to a potentially unbounded number of \pa-axioms. 

Loosely, a proof of (\ref{examples}) is called {\it compact} if the whole collection of reasonings ${\cal M}(n)$ leading to $F(n)$ for $n=0,1,2...$  fit into some finitely axiomatized fragment of \pa. Otherwise, the proof is called {\it non-compact}. 

It is obvious that any \pa-proof of a formula $\forall x F(x)$ naturally yields a compact proof of (\ref{examples}). 
A specific example of non-compact reasoning is given by the standard proof of Mostowski's reflexivity theorem, MRT, which states that \pa\ proves the consistency of each of its finitely axiomatized fragments.  

In Section~\ref{s3} we give the rigorous definition of compact and non-compact proofs within the framework of selector proofs.  

\section{Explicit reflection and selector proofs}

Selector proofs were introduced in [4] (as proofs of schemas) and in [5], cf. also \cite{SSK23}; they are a sound (and overlooked) addition to the existing rigorous proof systems. Selector proofs have been tacitly used in everyday mathematics but were not explored in mathematical logic. Acknowledging and studying selector proofs sheds new light on some foundational issues in mathematics, such as the (un)provability of consistency.

A {\it serial property ${\cal F}$ in a theory $T$} is an effectively specified collection $\{ F_0, F_1,\ldots,F_n,\ldots \}$ of formulas in the language of $T$. For example, ${\cal F}$ can be given by a primitive recursive enumeration of the codes of formulas $F_i$, for $i=0,1,2,\ldots$ .

A {\it selector proof} of a serial property ${\cal F}$ is a pair of
\renewcommand{\labelenumi}{\text{\roman{enumi})}}
\begin{enumerate}
\item  {\it selector}: a computable operation that given $n$ provides a proof of $F_n$ in $T$. For the foundational purposes, in this work, we assume that such an operation is primitive recursive and explicitly defined.\footnote{A more general view of selector proofs has been studied in \cite{Gad25}.}\item {\it verifier}: a proof in $T$ that the selector does (i). 
\end{enumerate}
The paper \cite{Art25} provides a body of examples that demonstrate that selector proofs have been tacitly adopted by mathematicians.  
\medskip\par

We naturally assume the standard G\"odel numbering as well as the usual formalization technique and notations within the limits of G\"odel and Feferman papers \cite{Fef60,God31}. In particular, we will not distinguish between a finite syntactic object $X$, its G\"odel number $\gn{X}$, and the corresponding \pa-numeral when safe.
\mpn

{\it A natural formalization in \pa\ of a given selector proof of ${\cal F}$ in $T$} is a pair $\langle s,v\rangle$ with: 
\renewcommand{\labelenumi}{\text{\roman{enumi})}}
\begin{enumerate}
\item an arithmetical primitive recursive term $s(x)$ formalizing the selector procedure, 
\item a natural formalization of a given verifier which is a \pa-proof $v$ of 
\[ \forall x [s(x)\!:_{T}\! F^\bullet(x) ].\] 
\end{enumerate}
Here ``$u\!:_{T}\!w$" is a short for the standard arithmetical primitive recursive formula {\it ``$u$ is proof of $w$ in $T$"} and $F^\bullet(x)$ is a natural primitive recursive term such that $F^\bullet(n) = \gn{F_n}$. This definition naturally extends from \pa\ to any other sufficiently strong coding theory. We also denote the arithmetical formula $(0=1)$ as $\bot$.

\subsection{Explicit reflection and the soundness of selector proofs}\label{s2.1}

The provability of {\it explicit reflection}: 
\begin{equation}\label{exref}
\mbox{\it for any natural number $m$ and formula $\varphi$, $\pa\proves\lc{m}\varphi \imp \varphi$}
\end{equation}
was first noticed by G\"odel in 1938 \cite{God38}. Still, his work remained unpublished until 1995, when the provability of explicit reflection had already been independently noted in \cite{AS92a,AS92b} and used in a similar situation in the then-emerging logic of proofs.\footnote{For further developments, cf. \cite{Art01a}.} Explicit reflection helps to establish the soundness of selector proofs with respect to the traditional single-formula arithmetical provability. 
\medskip

The following properties of selector proofs in \pa\ hold (\cite{Art25}):
\begin{itemize}
\item Selector proofs are finite syntactic objects. 
\item Selector proofs absorb conventional proofs as special cases.
\item The selector proof predicate in \pa\
\vskip-10pt
\[ \mbox{``$\langle s, v\rangle$ is a proof of a serial property ${\cal F}$"}
\]  
\vskip-8pt
is decidable. 
\item Selector proofs are as good as conventional proofs for individual formulas. Selector provability of a serial property $\{F_0,F_1,\ldots \}$ yields the conventional provability of each of $F_n$'s. Selector proofs do not add any new theorems to \pa.
\end{itemize}
The last item answers a popular question of whether \pa\ enhanced by selector proofs is a stronger theory that \pa. The answers is no, selector proofs do not increase the conventional deductive power of \pa. Here is the proof of this item.
\begin{proposition}(\cite{Art25}) {\it If a serial property ${\cal F}=\{F_0,F_1,\ldots \}$ is selector provable in \pa, then $\pa\proves F_n$ for each $n=0,1,2,\ldots$.}
\end{proposition}
\begin{proof} Suppose 
\[ \pa\proves\forall x [\lc{s(x)} F^\bullet(x) ].\] 
Then for each natural number (numeral) $n$, 
\[ \pa\proves \lc{s(n)} F^\bullet(n), \] 
i.e., 
\[ \pa\proves \lc{s(n)} F_n . \] 
Since $s(n)$ is a specific standard natural number, by the explicit reflection (\ref{exref}), 
\[ \pa\proves F_n . \] 
\end{proof}

\section{Compact and non-compact selector proofs}\label{s3}

To make the informal intuition of compact and non-compact proofs from Section~\ref{intro} mathematically rigorous, we need a framework in which individual proofs 
${\cal M}(n)$ of $F(n)$ for each $n$ are well-defined formal objects. Conveniently, selector proofs naturally provide such a framework. For our foundational purposes, we define the notions of compactness and non-compactness for formalizable selector proofs in \pa\ and leave their extension to broader classes of proofs to further studies. 
%For the rest of the paper, ``compact proof" will denote ``compact selector proof." 
\mpn

A selector proof of a serial property $\cal F$ with selector $s(x)$ is called {\it compact}, if for some finitely axiomatizable fragment $W$ of \pa,
\[ 
\pa\proves\forall x [s(x)\!:_{W}\! F^\bullet(x) ].\]  
A selector proof of $\cal F$ is {\it non-compact}, if it is not compact. 
For the rest of the paper, ``compact proof" and ``non-compact proof"  denote ``compact selector proof" and ``non-compact selector proof."
\mpn
Consider the well-known proof of the Mostowski reflexivity theorem MRT:
\begin{quote}
{\it Let $\pa_{\upharpoonright n}$ be the fragment of \pa\ with the first $n$ axioms. Then for each natural number $n$, $\pa\vdash \mbox{\sf Con}_{{\sf PA}_{\upharpoonright n}}.$}
\end{quote}
From this proof, we can extract a primitive recursive function (selector) $s(x)$ that given $n$ builds a derivation of $\mbox{\sf Con}_{{\sf PA}_{\upharpoonright n}}$ in \pa. Note that 
$\mbox{\sf Con}_{{\sf PA}_{\upharpoonright n}}$ implies ``{\it $D$ is not a proof of $\bot$}" for any specific derivation $D$ in $\pa_{\upharpoonright n}$.

\begin{Prop} $\ \ $ {\it The given proof of the Mostowski reflexivity theorem is a non-compact selector proof.}
\end{Prop}
\begin{proof} The fact that the proof of MRT is a selector proof is immediate from its format. An inspection of the proof reveals that it uses a well-known construction called \textit{partial truth definitions} which can be traced back to \cite{KHW55} 
(cf. also \cite{Buss98, HP17, vOo99, Pud98, Smo85}). Namely, for each $n=0,1,2,\ldots$ we build, in a primitive recursive way, an arithmetical formula 
\[ \mbox{\it Tr}_n(x,y), \]
called \textit{truth definition for $\Sigma_{n}$-formulas}, which satisfies natural properties of a truth predicate.
Intuitively, if $\varphi$ is a $\Sigma_{n}$-formula and $y$ is a sequence encoding values of the parameters in $\varphi$, then 
$\mbox{\it Tr}_n(\varphi,y)$ defines the truth value of $\varphi$ on $y$. 
%We drop ``$y$" and write $\mbox{\it Tr}_n(\varphi)$ when $\varphi$ does not have parameters. 
The following lemma is well-known in proof theory of \pa.
\begin{Lemma}\label{l1} {\it For any $\Sigma_{n}$-formula $\varphi$, the following Tarski's condition holds and is provable in \pa:
\begin{equation}\label{Tarski}
\mbox{\it Tr}_n(\varphi,y)\ \leftrightarrow\ \varphi(y).
\end{equation} 
In particular, \pa\ proves $\neg\mbox{\it Tr}_n(\bot)$.}
\end{Lemma}
The analysis of the proof of Lemma~\ref{l1} (cf. \cite{vOo99}) shows that the sets of induction instances involved grow with $n$, so this proof is not compact. 
\end{proof}

For further proof-theoretic analysis of MRT, cf. Section~\ref{s5}.

\section{Selector provability vs. single-formula provability}\label{s4}

Consider a serial property generated by a single arithmetical formula $F(x)$, 
\[ F(0), F(1), \ldots, F(n),\ldots .
\]
We call such property a {\it schema} and denote 
\[ \{F(x)\}. \] 
Obviously, if $\pa\proves \forall x F(x)$, then $\{F(x)\}$ is selector provable. Indeed, from the proof of $\forall x F(x)$ we immediately obtain a selector $s(x)$ and its verification in \pa, i.e. a proof of 
\[ \forall x [\lc{s(x)} F^\bullet(x)].
\]
This selector can be chosen to be compact: the amount of induction needed is limited by a finite sub-theory of \pa\ sufficient for proving $\forall x F(x)$. 

The following easy but fundamental observation made by Elijah Gadsby in \cite{Gad25} shows that the converse holds as well. 
\begin{Prop}\label{p4}  {\it If a schema $\{F(x)\}$ has a compact proof in \pa\ then $\pa\proves\forall x F(x)$.}
\end{Prop}
\begin{proof}
If  $\{F(x)\}$ has a compact proof in \pa, then there is a finitely axiomatized fragment $W$ of \pa\ such that
\[ 
\pa\proves\forall x [s(x)\!:_{W}\! F^\bullet(x) ].\] 
Therefore, 
\[ \pa\proves\forall x \Box_W F(x). \]
(here $\Box_W$ is the standard provability predicate in $W$). 
By the well-known \pa-provable uniform reflection principle for finitely axiomatized fragments $W$ of \pa, 
\[ \pa\proves\forall x F(x).\]
\end{proof}
\begin{Coro} Compact proofs of consistency are ruled out by G2.
\end{Coro}
Indeed, the \pa-consistency is a contentual universal statement 
\[ \mbox{\it ``for all natural numbers $n$, $\neg\lc{n}(0\!=\!1)$"}.\]
According to definitions, a selector proof of consistency is a selector proof of the schema $\{\neg\lc{x}\bot\}$ which is covered by Proposition~\ref{p4}. Since, by G2, \pa\ does not prove 
$\forall x \neg\lc{x}(0=1)$, the schema $\{\neg\lc{x}\bot\}$ and the consistency property, do not have compact selector proofs in \pa. \begin{Coro} Non-compact proofs of consistency are not ruled out by G2. 
\end{Coro}
There are (non-compact) selector proofs of consistency property, cf. \cite{Art20,Art25}. 
\mpn

All the above offer a well-principled answer to the question posed by Lev Beklemishev: what kinds of proofs are ruled out by G\"odel's Second Incompleteness Theorem? G2 prohibits compact proofs but does not rule out non-compact proofs of \pa-consistency.

\subsection{MRT and a proof of \pa-consistency in \pa}\label{s5}

 Now we answer a question by Mikhail Katz: whether MRT counts as a \pa-consistency proof of in \pa? 
 %Here is a complete account of MRT, its proof and formalizations. 
 The standard proof of MRT establishes that
 \begin{equation}\label{mrt}
 \mbox{\it for all $n$, $s(n)$ is a \pa-proof of ${\sf Con}_{{\sf PA}_{\upharpoonright n}}$. }
 \end{equation}
 for the primitive recursive selector $s(x)$ described above. This can be regarded as a selector proof of the consistency property of the form 

 \begin{equation}\label{mrtc}
 \mbox{\it for all $n$, ${\sf PA}_{\upharpoonright n}$ is consistent. }
 \end{equation} 
 However, (\ref{mrt}) it is not yet a proof in \pa. To make it a selector proof of \pa-consistency in \pa, we need to prove the direct formalization of (\ref{mrt}) in \pa: 
 \begin{equation}\label{mrtf}
\pa\proves \forall x\ \lc{s(x)}{\sf Con}^\bullet_{{\sf PA}_{\upharpoonright x}}.
% \mbox{\it for all $n$, $s(n)$ is a proof of ${\sf Con}_{{\sf PA}_{\upharpoonright n}}$. }
 \end{equation}
 This is how \pa-consistency was selector-proven in \pa\ in \cite{Art20,Art25}. 
 \mpn
 
 From (\ref{mrt}), we can also derive the ``official" formulation of MRT:
 \begin{equation}\label{omrt}
\mbox{\it for all $n$, there is a proof of ${\sf Con}_{{\sf PA}_{\upharpoonright n}}$ in \pa,}
 \end{equation}
 and its formalized version
 \begin{equation}\label{omrtf}
\pa\proves \forall x\ \Box{\sf Con}^\bullet_{{\sf PA}_{\upharpoonright x}}
 \end{equation}
 (here $\Box$ is the standard provability predicate in \pa), 
which does not count as a proof of \pa-consistency. For example, if (\ref{omrtf}) were a proof of \pa-consistency, it would yield a \pa-proof of individual 
${\sf Con}_{{\sf PA}_{\upharpoonright n}}$'s, but this conclusion relies on the implicit reflection
\[ \Box F \imp F \]
which is not provable in \pa.

\section{Discussion}

There is a robust intuition of compact and non-compact reasoning, Section~\ref{intro}. To make this intuition precise, we need some rigorous representation of arithmetical reasoning $R(n)$ leading, given $n$, to the conclusion $F(n)$ in a proof of (\ref{examples}). By the formalization principle, such $R(n)$ could be represented by a formal \pa-derivation of $F(n)$. A standard G\"odelian formalization leads directly to the selector proofs format
\[ \lc{R(n)}F(n).\] 
%without assuming that $R(n)$ is uniformly computable from $n$, as required in selector proofs. However, the format of selector proofs is warranted. 
%We certainly do not assume that there are no proofs other than selector proofs. 
Then we project the idea of compact/non-compact proofs to the existing selector proofs framework, enabling us to speak about compact and non-compact selector proofs. However, the concept of compact/non-compact proofs is applicable in greater generality, e.g., when $R(n)$s are not assumed to be computable.

Due to Proposition~\ref{p4}, the provability of $\forall x F(x)$ is mathematically equivalent to having a compact selector proof of the schema $\{F(x)\}$. This enables us to answer to Beklemishev's question, ``What proofs of consistency are prohibited by G2?": compact selector proofs. 

Interestingly enough, the notion of selector proofs is not present in the question, but is dominant in the exact answer. The reason for this is intrinsic: conventional provability of  $\forall x F(x)$ is equivalent to existence of a compact selector proof of the schema $\{F(x)\}$.

%Furthermore, G2 sheds no light on other potential proofs of consistency. 
Selector proofs of \pa-consistency in \pa\ \cite{Art20,Art25} complement this answer: a step from compact to non-compact proofs yields a formalizable in \pa\ proof of \pa-consistency.

%To detect and describe the compactness/non-compactness phenomenon, we have to focus on how a proof procedure given $n$ reaches the conclusion $F(n)$ in a proof of {\it for all $n$, $F(n)$}. A natural framework for such an analysis is provided by selector proofs with terms $P(n)$'s for individual proofs of $F(n)$ given $n$. So, for understanding the real limitations of G2 for consistency proofs, we need the selector proofs framework. 

Non-compact proofs constitute a class of mathematical proofs overlooked in foundational studies and not covered by G\"odel's Second Incompleteness theorem. 

\section{Thanks}\label{s6}

Many thanks to participants of inspiring recent  discussions on these issues: 
\medskip\par
Arnon Avron, Alexandru Baltag, Lev Beklemishev, Thierry Coquand, Nachum Dershowitz, Mel Fitting, Elijah Gadsby, Joel Hamkins, Richard Heck, Mikhail Katz, Vladimir Krupski, Menachem Magidor, Anil Nerode, Dmitry Nogin, Elena Nogina, Vincent Peluce, Andrei Rodin, Ilya Shapirovsky, Wilfried Sieg, Albert Visser, Noson Yanovsky, and many others.

%Non-compact proofs is a class of arithmetical first-order reasoning which is used in mathematics but overlooked in the analysis of (un)provability of consistency. They are not in the scope of G\"odel's Second Incompleteness Theorem. Moroever, when non-compact proofs are taken into account, the consistency of a theory containing arithmetic becomes provable in this theory. 

\end{document}